\documentclass[11pt]{amsart}
\usepackage{amsmath}
\usepackage{amsfonts}
\usepackage{graphicx}
\usepackage{color}

\usepackage{amssymb}
\usepackage{graphics}
\usepackage{amsmath}
\usepackage[english]{babel}
\usepackage{hyperref}
\usepackage{bbm}
\usepackage{yfonts}
\usepackage{mathrsfs}

\newcommand{\remove}[1]{}
\newcommand{\ignore}[1]{}

\newtheorem{theorem}{Theorem}[section]

\newtheorem{corollary}[theorem]{Corollary}

\newtheorem{lemma}[theorem]{Lemma}
\newtheorem{definition}[theorem]{Definition}

\newtheorem{notation}[theorem]{Notation}
\newtheorem{claim}[theorem]{Claim}

\newtheorem{observation}[theorem]{Observation}
\newtheorem{remark}[theorem]{Remark}

\newcommand{\bae}{\begin{equation}\begin{aligned}}
\newcommand{\eae}{\end{aligned}\end{equation}}

\newcommand{\constD}{\phi}

\newcommand{\constK}{\Gamma}

\newcommand{\graphA}{G}
\newcommand{\graphB}{G_c}
\newcommand{\graphC}{G_{1}}
\newcommand{\graphD}{G_{2}}

\newcommand{\tm}{T(M)}
\newcommand{\tmgamma}{T(M,\Gamma)}
\newcommand{\la}{\mathbb{Z}^2 \left(\frac{1}{M}\right)}
\newcommand{\area}{\textrm{area}}
\newcommand{\wlogg}{\textrm{without loss of generality}}

\begin{document}

\title{Edge-Removal and Non-Crossing Perfect Matchings}

\author{Aviv Sheyn  \and Ran J. Tessler}

\thanks{We would like to thank  Micha Asher Perles for introducing us the subject, and for many helpful discussions and ideas.}

\maketitle

\begin{abstract}
We study the following problem - How many arbitrary edges can be
removed from a complete geometric graph with $2n$ vertices such that
the resulting graph always contains a perfect non-crossing matching?
We first address the case where the boundary of the convex hull of the original graph
contains at most $n+1$ points. In this case we show that $n$ edges can be removed, one more than the general case. In the second part we establish a lower bound for the case where the $2n$
points are randomly chosen. We prove that with probability which tends to
$1$, one can remove any $n+\Theta (\frac{n}{\log {(n)}})$ edges but the residual graph
will still contain a non-crossing perfect matching.
We also discuss the upper bound for the number of arbitrary edges one must remove in
order to eliminate all the non-crossing perfect matchings.

\end{abstract}

\section{Definitions, Motivation and Background}
\begin{definition}
A geometric graph $(G,i)$ is an embedding $i : G\mapsto\mathbb{R}$ of a graph $G=(V(G),E(G))$ in the plane such that the vertices (the elements of $V(G)$) are mapped to points in general position, and the edges (the elements of $E(G)$) are mapped to the the straight line segments connecting them. Those segments may intersect.
\end{definition}
For convenience we shall not mention the embedding $i$ and consider the graph $G$ as its embedding in the plane.

\begin{definition}
As usual, we say that $H = (V(H),E(H))$ is a (geometric) subgraph of $G$ if $V(H)\subseteq{V(G)}$ and $E(H)\subseteq{E(G)}$.
For a subgraph $H$ we denote by $G-H$ the graph $(V(G), E(G)-E(H))$.
\end{definition}
\begin{notation}
Let $l$ be an oriented line in the plane. We denote by $\emph{r}(l), \emph{l}(l)$ the halfplanes right of $l$ and left of $l$ with respect to the chosen orientation. In addition, given an imbedded planar graph $\graphA$, we denote and  by $LP(l),RP(l)$ the sets of points in $V(\graphA)$ which lie to the left of $l$ and to the right of $l$ ,respectively. We omit $\graphA$ from the notation, as it will usually be clear to which graph are we referring.
\end{notation}
\begin{definition}
Given a geometric graph $G=(V,E)$, we define a $\emph{Perfect}$ $\emph{Non Crossing Geometric Matching}$ as a subset $M\subseteq{E}$ of non-intersecting edges such that any vertex of $G$ belongs to exactly one edge in $M$.
\end{definition}
As we shall use this term many times in this article, we shall often call a perfect non-crossing perfect geometric matching simply a \emph{matching}.
\begin{notation}
Let $\graphA$ be a complete geometric graph on $2n$ vertices. Denote by $h(\graphA)$
the largest number with the property that for any set of $h(\graphA)$ edges, if we remove them from the graph, we can still find a perfect matching in the resulting graph.
\end{notation}

Of course, $h(\graphA)$ depends on the positions of the vertices of $\graphA$ in the plane. If the $2n$ points are in convex position then it is easy to see that we have $n$ disjoint matchings, so the removal of $n-1$ edges must leave a matching. On the other hand it is also easy to see that if we remove $n$ adjacent boundary edges from $\graphA$, no matching remains.
In \cite{KE} the collection of all sets of $n$ edges whose removal eliminates every matching is fully characterized.
So for any $\graphB$ with vertices in convex position, we have $h(\graphB)=n-1$.

The question of how many edges is it possible to remove without eliminating all the matchings in the general case, when the points are not in convex position, was solved in \cite{PE} and \cite{AOCS}. For the sake of completeness, we provide their solution and proof in the next section, as we generalize those methods in later sections of this paper.

Questions regarding matchings and other combinatorial objects in geometric graphs are rather natural and appear both in theoretical papers in combinatorial geometry and in papers regarding geometric computation. A good source for open questions in that field and some recently proved results is the recent paper \cite{AOCS}.

One of the most important tools we use is the combinatorial version of the Ham Sandwich Theorem:

\begin{theorem} \label{thm:ham} $[Ham ~ Sandwich]$
Let $X$ and $Y$ be two sets of red and blue points, respectively, in general position in the plane. There exists a line $l$ such that $|X \cap{\emph{r}(l)}|=|X \cap{\emph{l}(l)}|$ and $|Y \cap{\emph{r}(l)}|=|Y \cap{\emph{l}(l)}|$. Moreover, if either $|X|$ or $|Y|$ is even, we may also assume that $l$ is not parallel to any segment which connects two points in the union of $X$ and $Y$.
\end{theorem}
For a proof see \cite{MA}.

The following consequence of this theorem will be of great importance in the forthcoming sections:

\begin{theorem} \label{thm:zvuvonim}
For any set of $n$ red points and $n$ blue points in general position in the plane, there exists a non-crossing perfect matching such that each edge of the matching has one blue vertex and one red vertex.
\end{theorem}
This can be proved by repeated use of the Ham Sandwich theorem, yielding a partition of the plane into convex regions, each of which contains exactly one blue and one red point.

\section{The General Case}
\label{sec:general case}
We now return to the question, in the general case, where the points are not necessarily in convex position, how many edges is it always possible to remove?
As said before, this problem was solved in \cite{PE} and \cite{AOCS}. Later we will considerably extend the method used in the proof of the result in the general case in order to establish one of our two main theorems, and therefore, for the sake of completeness we shall present the original claim and its proof.

\begin{theorem} \label{thm:general}
Given a set of $2n$ points in general position in the plane, let $\graphA$ denote the complete graph on this set.  Then the removal of any $n-1$ edges from $\graphA$ will always leave a matching.  This result is tight with respect to the number of edges removed.
\end{theorem}

\begin{proof}
Let $S$ be a set of at most $n-1$ edges, $H$ be the subgraph of $G$ containing these edges and their vertices.
Denote by $V_1, V_2, ... , V_m$ the vertex sets of the components of $H$ ordered by size (in decreasing order).
We claim that it is possible to partition the $2n$ vertices of $\graphA$ into two sets of $n$ points each such that the vertex set of each connected component of $H$ is contained in one of the sets.
If $\sum_{j\leq{m}}|V_j|\leq{n}$ then our claim is clearly true. Otherwise let $i$ be such that $\sum_{j\leq{i}}|V_j|\leq{n}$ and  $\sum_{j\leq{i+1}}|V_j|>n$. Such an $i$ exists because $H$ has at most $n-1$ edges, hence $|V_1|\leq{n}$. \\
Note that if some component of $H$ has $k$ vertices then it contains at least $k-1$ edges, leaving at most $n-k$ edges in the remaining components of $H$.  These edges can cover at most $2(n-k)$ vertices of the $2n-k$ remaining vertices in $G$, so there must be at least $k$ vertices in $\graphA$ that are not in $H$. Therefore there is a set of $|V_{i+1}|$ vertices of $\graphA$ not in $H$ from which we can choose vertices which combined with the sets $V_1, V_2, ... , V_i $  make exactly $n$ vertices (this is possible since by our choice of $i$, $|V_{i+1}|>n-\sum_{j\leq{i}}|V_j|$).  Thus we have a set of $n$ points which satisfies our claim.
We now have two sets of $n$ points, and each edge of $H$ has both endpoints in one set. By Theorem \ref{thm:zvuvonim} we have a matching between the two sets and this matching does not contain any edge of $H$.
\end{proof}
\begin{remark}
We noted already that in the convex case we can remove $n$ edges and eliminate all matchings so the result is tight with respect to the number of edges in $H$.\\
\end{remark}
What happens if the points are not in convex position - is it always possible to remove $\emph{more}$ than $n-1$ edges and still have a matching?
In \cite{AOCS} it is conjectured that for $k\geq{n-2}$, if a set of $2n$ points has at least $2n-k$ points on the boundary of its convex hull then it is possible to remove $k+1$ edges and still have a matching. The case of $k=n-2$ follows from the theorem above. We now prove our first main theorem which is the case where $k=n-1$.

\begin{theorem}\label{thm:hGc1isn}
Starting from the complete graph on $2n$ vertices in the plane in general position whose convex hull's boundary contains at most $n+1$ vertices, we can remove any $n$ edges and still have a matching.
\end{theorem}
\begin{remark}
There are geometric graphs for which this result is tight.
\end{remark}
The idea of the proof will be to show first that any subgraph $H$ with $n$ edges that eliminates all matchings (i.e., there is no matching in $G-H$), must have a certain structure and then to use the condition on the boundary of the convex hull to show that for any complete geometric graph $\graphC$ which satisfies this condition such a subgraph does not exist.

\begin{lemma} \label{lem:threeparts}
Let $\graphA$ be a complete geometric graph on $2n$ points in the plane in general position, and let $H$ be a subgraph of $\graphA$ with $n$ edges such that there is no matching in $\graphA-H$.  Then:\\
$(a)$ $H$ is a tree.\\
$(b)$ All vertices of degree at least $2$ in $H$ are on the boundary of the convex hull of $\graphA$.\\
$(c)$ If we color the vertices of $H$ by red and those vertices not in $H$ by blue, then for any edge $e$ of $H$, let $l_e$ be the supporting line (the line through the vertices of $e$).
Then on each side of $l_e$ the number of red points is equal to the number of blue points.
\end{lemma}
We first prove part $(a)$.
\begin{proof}
\emph{[Part $(a)$]} First, we show that $H$ must be connected.
If $H$ is disconnected then we denote by $V_1, V_2, ... , V_k $ the sets of vertices of the components of $H$ (in order of decreasing size). In this case $V_1\leq{n}$ so we can continue in the same manner as in the proof of Theorem \ref{thm:general}, dividing the vertices into two classes and finding a matching. So $H$ must be connected.\\
Second, we show that $V(H)\geq{n+1}$. If $V(H)\leq{n}$, then by Theorem \ref{thm:zvuvonim} we can find a matching which avoids edges of $H$, and this contradicts the assumption that $H$ eliminates all matchings.\\
Thus, $H$ is a connected graph with $n$ edges and at least $n+1$ vertices, and therefore must be a tree.\\
\end{proof}
In particular we have $V(H)=n+1$ . Color each vertex of $H$ by red and the remaining $n-1$ points color blue. \\
We shall call a set of points $X\subseteq{V(\graphA)}$ \emph{matchable} if we can find a perfect matching on $X$ without using edges of $H$.

We shall now make some preparations towards the proof of the rest of the Lemma. We shall be using the following easy facts that we have put together in one observation:
\begin{observation}
Every set with the same number of red and blue points (according to our coloring) is matchable, according to Theorem \ref{thm:zvuvonim}. A set with an even number of points which has more blue points than red points is matchable. The union of disjoint convex regions such that the points in each region are matchable is itself matchable.
\end{observation}

Our next step is to show that each leaf of $H$ is connected to the boundary of the convex hull by an edge of $H$.
Let $v$ be a leaf vertex of $H$. Draw any line $l$ through $v$. Choose a direction on $l$. Let $D(l)$ be the difference between the number of red points and blue points on the right side of $l$ relative to the chosen direction. Suppose $D(l)=i$ . We start rotating the line with $v$ as center. Whenever $l$ meets some point of $\graphA$, $D(l)$ changes by $\pm{1}$ because the points are in general position. But when we get back to the original line $l$ with the opposite direction we have $D(l)=1-i$, because except $v$ we have $n-1$ blue points and $n$ red points. Thus at some line $l_0$ we have $D(l_0)=0$; we may assume that $l_0$ contains another point $u \neq v$ of $\graphA$, since we can always rotate the line $l_0$ to the left until it meets some new point. \\
We have two cases :\\
$(i).~$ $u$ is a red point.[see figure 1]\\
$(ii).~$ $u$ is a blue point. [see figure 2]
\begin{claim}
In case $(i)$, $u$ lies on the boundary of the convex hull.
\end{claim}
\begin{proof}
In this case, on each side of $l_0$ there are the same number of red and blue points. If $uv$ is not an edge of $H$ then we can match $u$ and $v$, $LP(l_0)$ and $RP(l_0)$ are both matchable convex sets, and so we have a perfect matching. Thus we can assume $uv$ is an edge of $H$. We want to show that $u$ lies on the boundary of the convex hull. Assume this is not the case. Then we can find two points, $x \in RP(l_0)$ and $w \in LP(l_0)$,  such that $xvwu$ is a non-convex quadrilateral and there are no points which are both left of $ux$ and right of $uw$ [see figure 1]. If $ux$ is not an edge of $H$ then we can match $u$ and $x$, and thus $(RP(l_0)-\{x\})\cup\{v\}$ and $LP(l_0)$ are disjoint convex matchable sets which together with $ux$ form a perfect matching. Indeed we can add $v$ to the first set as a blue point since it is a leaf of $H$ and not connected to any point in this set. We can assume therefore that $ux$ is an edge of $H$. In the same way we may assume that $uw$ is an edge of $H$, but $H$ is a tree so $xw$ is not an edge of $H$ and therefore can be matched. But then $(RP(l_0)-\{x\})\cup\{v\}$ and $(LP(l_0)-\{w\})\cup\{u\}$ are disjoint convex matchable sets because we removed a red point from each set and added one point to each set. Thus we have a perfect matching in any case, proving by contradiction that $u$ lies on the boundary of the convex hull.
\end{proof}

Before we analyze case $(ii)$ we prove the next claim:
\begin{claim} \label{claim:case1}
If $v$ is on the boundary of the convex hull and $u$ is not, then only Case $(i)$ is possible.
\end{claim}
\begin{proof}
First note that the points on the boundary of the convex hull next to $v$ must be blue. Otherwise if $v'$ is some red point next to $v$ in the boundary of the convex hull, we have two possibilities: Either $vv'$ is an edge of $H$, or $vv'$ not an edge of $H$. In the former possibility $v'=u$ so $u$ is on the boundary of the convex hull contrary to assumption. In the latter possibility  we can match $vv'$, and the remaining $n-1$ red and $n-1$ blue points are matchable, so we have found a perfect matching contrary to assumption. Thus we conclude that $v'$ must be blue.

As before we start rotating $l$ with $v$ as the center. We can start from a position where all the points are to the left of $l$. Rotating $l$ to the left we first meet a blue point so we start with $D(l)=-1$, and we end with a blue point so one step before that we have $D(l)=2$.  Thus on the way we must move one time from 0 to 1, meaning that we have at some point an equal number of red and blue points right of $l$ and then we meet a new red point; this is exactly case $(i)$.
\end{proof}

Now we can analyze case $(ii)$:
\begin{claim}
In case $(ii)$, $u$ lies on the boundary of the convex hull.
\end{claim}
\begin{proof}
In this case, in $RP(l_0)$ there are the same number of red and blue points, and in $LP(l_0)$ there are two more red points than blue points. If the neighbor of $v$ in $H$ is in $RP(l_0)$  then both $RP(l_0)$ and $LP(l_0)\cup\{u,v\}$ are matchable convex sets so we have a perfect matching. Indeed, we add two blue points to the left side because $v$ is connected only to the right side. Thus we can assume that the neighbor of $v$ lies in $LP(l_0)$. Now we look at  $LP(l_0)\cup\{u,v\}$. It is a set with two more red points than blue points. $v$ is on the boundary of the convex hull of this set so by the claim, the neighbor of $v$ in $H$ must be on the boundary of the convex hull of $LP(l_0)\cup\{u,v\}$. Call this point $w$. If $w$ is on the boundary of the original convex hull the claim follows. Otherwise, we have the case shown in [figure 2]. From the figure it is clear that $w$ can be matched to some point $x$ belonging to the convex hull 
of $RP(l_0)$, and thus $(RP(l_0)-\{x\})\cup\{v\}$ and $(LP(l_0)-\{w\})\cup\{u\}$ are both matchable disjoint convex sets so we have a perfect matching (Note that we removed a red point from the left side and added a blue one so indeed we have a matchable set).
\end{proof}
From both lemmas we obtain the following claim:
\begin{claim}\label{clm:leaf}
The neighbor of every leaf in $H$ lies on the boundary of the convex hull.
\end{claim}

The following claim will be used in the proof of part $(c)$ of the lemma.
\begin{claim} \label{clm:hezi}
Let $e=uv$ be an edge in $H$ such that $v$ is of degree at least $2$ and lies on the boundary of the convex hull. Then on each side of the line extending $e$, the number of red points equals the number of blue points. [see figure 3]
\end{claim}
This would prove part $(c)$ if we knew that all interior vertices of $H$ are of degree $1$, i.e. if we knew $(b)$.
\begin{proof}
First note that the neighbors of $v$ on the boundary of the convex hull are red also, otherwise we can match one of them with $v$ and reduce the problem to $2n-2$ points and $n-2$ edges removed, and by Theorem \ref{thm:general} we can find a perfect matching in this reduced graph. \\
We want to show that on each side of $e$ there are equal numbers of red and blue points. If this is not the case, assume, $\wlogg$, that to the right of $e$ there are more blue points.
Take a line $l$ and start rotating it with $v$ as the center of rotation. We start with the $uv$ direction and rotate to the right side. In the beginning $D(l)<0 $ and at the end (when $l$ goes through both $v$ and some adjacent boundary point) we have $D(l)=0$. since this quantity changes only by $1$ or $-1$ at each stage, we conclude that for some line $l_0$ we have $ D(l_0)=-1 $  Now $RP(l_0)\cup\{v\}$ is a matchable set with the same number of red and blue points, and $LP(l_0)$ is also a matchable set because in $LP(l_0)$ we have two more red points than blue ones.  But $v$ was joined to $2$ points from $LP(l_0)$, namely $u$ and the neighbor on the boundary of the convex hull [see figure 3]. Initially $(H\cap{LP(l_0)})\cup\{v\}$ is a tree, so if we remove $v$ (which is of degree at least $2$), the "red-graph" $(H\cap{LP(l_0)})$ is not connected and has two more red points than blue points; hence by part $a$ of our proposition there must be a matching of $LP(l_0)$. Together with $RP(l_0)\cup\{v\}$ we thus obtain a perfect matching. By contradiction, we conclude that on each side of $e$ there is an equal number of red and blue points. \\
\end{proof}
\begin{claim}
If $v$ is on the boundary of the convex hull, $u$ is of degree at least $2$ in $H$, and $vu$ is an edge of $H$, then $u$ must also lie on the boundary of the convex hull.
\end{claim}
\begin{proof}
Assume $u$ is not on the boundary of the convex hull. We saw that every leaf is connected to the boundary of the convex hull so $v$ must be of degree at least $2$. Denote by $l_0$ the line through $u$ and $v$. By the previous claim, on each side of $l_0$ there is an equal number of red and blue points. Because $u$ is an interior point we can find two points - $x \in RP(l_0)$ and $w \in LP(l_0)$, such that $xvwu$ is a non-convex quadrilateral and there are no points which are on the left of $ux$ and on the right of $uw$ [see figure 1].\\
We analyze several cases:\\
$(i).~$ $ux$ is not an edge of $H$ and there is a neighbor (in $H$) of $u$ in $RP(l_0)$. In this case we can match $u$ and $x$. Of course $LP(l_0)$ is a matchable set and so is $(RP(l_0)-\{x\})\cup\{v\}$ because we have at most two more red points than blue points, but the graph of $H$ in this set is not connected so it is matchable by the same argument as in the proof of the previous claim. Thus we have a perfect matching, and hence a contradiction.
$(ii).~$ $uw$ is not an edge of $H$ and there is a neighbor of $u$ (with respect to $H$) in $LP(l_0)$. We proceed as in case $(i)$.\\
$(iii).~$  $u$ has a neighbor in $LP(l_0)$ and also in $RP(l_0)$, $uw$ and $ux$ are both edges of $H$. In this case we match $x$ and $w$. This is possible because $H$ has no cycle. Then $(RP(l_0)-\{x\})\cup\{v\}$ and $(LP(l_0)-\{w\})\cup\{u\}$ are matchable sets and we have a perfect matching.  Again we have reached a contradiction.\\
$(iv).~$ $u$ has a neighbor in $LP(l_0)$ and not in $RP(l_0)$.  $uw$ is an edge and $ux$ is not. If $xw$ is not an edge of $H$ we match $x$ and $w$, so $(LP(l_0)-\{x\})\cup\{v\}$ is a matchable set and $(RP(l_0)-\{w\})\cup\{u\}$ is also matchable because $w$ is a red point. If $xw$ is an edge, $ux$ must be not be, so we can match $u$ and $x$ and because now $x$ is a red point we have that $(RP(l_0)-\{x\})\cup\{v\}$ and $LP(l_0)$ are both matchable. So in every case we have a perfect matching, and again we have reached a contradiction.\\
$(v).~$ $u$ has a neighbor in $RP(l_0)$ and not in $LP(l_0)$ ($ux$ is an edge and $uw$ is not, with respect to $H$). This case is similar to case $(iv)$.

Now, because $u$ is of degree at least $2$, one of these cases must occur, but we have shown that each case leads to a contradiction.  Thus our initial assumption that $u$ is not on the boundary of the convex hull is impossible, and the claim is proved.
\end{proof}
Now we have everything we need for deducing parts $(b),(c)$ and thus finishing the proof of Lemma \ref{lem:threeparts}.
\begin{proof}
\emph{[Lemma \ref{lem:threeparts}, parts $(b),(c)$]}
We know from Claim \ref{clm:leaf} that every leaf has a neighbor in the boundary of the convex hull. In addition we know from the last claim that there are no interior points which are not leaves that are connected to the boundary of the convex hull. Thus every vertex of degree at least $2$ in $H$ must lie on the boundary of the convex hull, and this proves part $(b)$.\\
Part $(c)$ of the lemma follows by Claim \ref{clm:hezi} (for boundary edges in $\graphA$ which are edges of $H$, the condition applies trivially), together with part $(b)$, as every edge in $H$ has at least one vertex on the boundary of the convex hull.
\end{proof}

We have characterized the family of graphs with $n$ edges whose removal can eliminate all matchings. Does every such graph eliminate all matchings? In the convex case the answer is yes (see \cite{KE}). In the general case it seems that some additional conditions need to be added.

Now we consider the case where the boundary of the convex hull contains at most $n+1$ points.We shall prove the theorem by induction.\\
As a basis for induction we will first check the case $n=2$.
In this case we have a non-convex quadrilateral and it is easy to see that we have $3$ disjoint matchings, so these cannot be eliminated by removing only $2$ edges.  Thus Theorem \ref{thm:hGc1isn} applies in this case. \\
We prove three claims. In all of them we assume that we are given $2n$ points in general position whose convex hull's boundary contains at most $n+1$ of them. We assume that $H$ is a graph of $n$ edges that eliminates all matchings, that is, there is no perfect matching in $\graphA-H$. We already know that $H$ has the structure of Lemma \ref{lem:threeparts}.
\begin{claim}
If $H$ contains an edge between two points on the boundary of the convex hull, it must be a boundary edge.
\end{claim}
\begin{proof}
We prove the claim by induction on $n$. For $n=2$ such an $H$ does not exist, so the lemma is true. Assume that Theorem \ref{thm:hGc1isn} is true for $m<n$.\\
Assume that we have an edge $e$ connecting two points in the boundary of the convex hull. Let $l_e$ be the line supporting $e$.  By Lemma \ref{lem:threeparts} we know that on each side of $l$ there are an equal number of red and blue points. Let $x$ and $y$ be the number of boundary points in $RP(l)$ and $LP(l)$, respectively. By assumption, $x+y+2\leq{n+1}$.\\
If there are at least $x$ points in the interior of $RP(l)$, then $RP(l)\cup{\{u,v\}}$ satisfies the conditions of the induction hypothesis, because we have two red points more than blue points in this subgraph, and the red segments in this subgraph form a subtree. (If $|RP(l)\cup{\{u,v\}}|=2n$ then $uv$ is a boundary edge and we are finished.)  Thus $RP(l)\cup{\{u,v\}}$ and $LP(l)$ are disjoint matchable sets and we get a perfect matching, contrary to assumption.\\
We conclude that there are at most $x-1$ interior points in $RP(l)$.  Similarly there are at most $y-1$ interior points in $LP(l)$ so altogether there are at most $x+y-2\leq{n-3}$ points in the interior.  This contradiction proves the claim.
\end{proof}

\begin{claim} \label{clm:cross}
For every leaf edge $e$ of $H$, let $l_e$ be the supporting line of $e$. Then the only edge of $H$ that intersects the line $l_e$ is $e$.
\end{claim}

\begin{proof}
We already saw that on each side of $e$ there are an equal number of red and blue points. Let $|LP(l_e)|=2x$ and $|RP(l_e)|=2y$. If there is no perfect matching then $LP(l_e)\cup{e}$ cannot be a matchable set because $RP(l_e)$ is a matchable set, so in $LP(l_e)\cup{e}$ there must be at least $x+1$ edges of $H$ (otherwise, by an inductive argument similar to that of the previous proof, $LP(l_e)\cup{e}$ would be a matchable set). Similarly, in $RP(l_e)\cup{e}$ there are at least $y+1$ edges. This means that in total we have at least $x+y+2-1=n$ edges.  Thus we have counted all edges of $H$ by considering only edges that do not cross $l_e$!  We conclude that no edge other than $e$ intersects $l_e$.
\end{proof}

\begin{claim} \label{clm:2 lines}
(a) Let $v_{0}v_{1}...v_{k}$ be $k+1$ adjacent points on the boundary of the convex hull such that $v_{0},v_{k}$ are of degree at least $3$ in $H$ and the rest are of degree $2$. Denote by $l_{1},l_{2}$ lines passing through $v_{0},v_{k}$ and their neighbors of degree $1$ in $H$ (they must have such!), $w_0$, and $w_k$, respectively. Then there must be a blue point $v$ on the boundary of the convex hull different from $\{ v_{i}\} _{i=0,1..,k}$ between the lines $l_{1}$ and $l_{2}$.\\
(b) The conclusion of part $(a)$ is also true when $v_0$ or $v_k$ has a neighboring leaf on the boundary or when $k=0$ if $v_0$ has $2$ neighboring 
 leaves in $H$.
\end{claim}
\begin{proof}
The lines $l_{1}$ and $l_{2}$ can cross inside the convex hull or outside it so we have two cases to consider [see figure 4].\\
$(i).$ The lines intersect outside the convex hull.\\
Without loss of generality assume $v_0$ is to the left of $v_k$.
We claim that in $RP(l_1)\cap{LP(l_2)}$ there are $k+1$ blue points and no red points, except for $v_{0},v_{1},...,v_{k}$:\\
By Claim \ref{clm:cross}, no edge of $H$ can cross either of the lines $l_1$ and $l_2$, and the graph of $H$ is connected, so there are no red points. Since in $RP(l_2)$ and $LP(l_1)$ there are equal numbers of red and blue points, by Claim \ref{clm:hezi} we must have in $RP(l_1)\cap{LP(l_2)}$ exactly $k+1$ blue points.\\
Let $x_1, x_2$ be adjacent points on the boundary of the convex hull such that $x_1x_2$ crosses both $l_1$ and $l_2$. Such points exist, for otherwise we would have a blue point on the boundary of the convex hull between $l_1$ and $l_2$ and we are finished. Also, $x_1x_2$ is not an edge of $H$ by Claim \ref{clm:hezi}.\\
Now
$$LP(l_1)\cup\{w_0\}-\{x_1\}~~,~~RP(l_1)\cap{LP(l_2)}\cup\{v_0,v_k\}~~,~~RP(l_2)\cup\{w_k\}-\{x_2\}~~,~~\{x_1,x_2\}$$
is a decomposition of the $2n$ points into four convex matchable disjoint sets. This decomposition defines a perfect matching. \\
(Note that $u_0$ and $u_k$ become blue points since we disconnect them from their neighbors). Thus we have reached a contradiction.\\
\\
$(ii).$ The lines intersect inside the convex hull.\\
In this case the region between lines $l_1$ and $l_2$ is divided into two subregions:\\
Denote $D_1=RP(l_1)\cap{LP(l_2)}$ and $D_2=LP(l_1)\cap{RP(l_2)}$. As in case 1, all points in $D_1$ and $D_2$ must be blue (except for the boundary points $v_0,..,..v_k$). It is also easy to check that the number of blue points in $D_1$ minus the number of blue points in $D_2$ is $k+1$, so in $D_1$ there are at least $k+1$ points. \\
Let $l_3$ be a line passing through $v_0$, such that in $RP(l_3)\cap{LP(l_2)}$ there are exactly $k+1$ blue points. (rotate $l_1$ to the right and use the fact that there are at least $k+1$ blue points in $RP(l_1)\cap{LP(l_2}$).\\
Let $x_1, x_2$ be adjacent points on the boundary of the convex hull such that $x_1 x_2$ crosses $l_2$.\\
Now
$$LP(l_3)\cap{LP(l_2)}-\{x_1\}~~,~~RP(l_3)\cap{LP(l_2)}\cup\{v_0\}~~,~~RP(l_2)\cup\{w_k\}-\{x_2\}~~,~~\{x_1,x_2\}$$
is a decomposition of all points into disjoint convex matchable sets which forms a perfect matching. Again we have reached a contradiction.

To conclude case $(ii)$ is not possible, and case $(i)$ is possible only if there are blue points between the lines on the boundary of the convex hull.\\
The proof of $(b)$ is very similar and uses the same ideas, so we omit it.
\end{proof}

Finally we can prove Theorem \ref{thm:hGc1isn}:

\begin{proof}
\emph{[Theorem \ref{thm:hGc1isn}]}
We now have severe restrictions on the geometric structure of $H$, due to Lemma \ref{lem:threeparts} and the last three claims (see [Figure 5] for a schematic illustration).

If $k$ edges of $H$ lie on the boundary of the convex hull, then the rest are $n-k$ leaf edges. The lines supporting the leaf edges divide the plane into $n-k+1$ regions. On the boundary of the convex hull there are $k+1$ red points, so we have at most $n-k$ blue points on the boundary of the convex hull.  Thus at least one region lacks blue points from the convex hull, so by Claim \ref{clm:2 lines} we reach a contradiction.
\end{proof}

We remarked, after stating Theorem \ref{thm:hGc1isn} that the result is tight with respect to both conditions - the number of points on the boundary of the convex hull and the number of edges removed. We shall now give examples that show it.\\
\emph{Example 1}:\\
Take $n+1$ points in convex position to form an ($n+1$)-gon $C$.  Put $n-1$ points inside $C$ such that all of them are very close to the center of some edge and such that they are on the same side of each diagonal of $C$. Now let $H$ be the boundary cycle of $C$.  $H$ has $n+1$ edges.  Each perfect matching must match at least two points of $C$ to each other; denote this couple $ab$.  Then on some side of $ab$ there are only points from $C$, and we must match some boundary edge which is in $H$.  Thus there are no matchings in $\graphA-H$.\\
\emph{Example 2}:\\
Let $D$ be a convex polygon with $n+2$ vertices $P_0, P_1, P_2, ..., P_{n+1}$. Put $n-2$ points inside $D$ close to $P_0$.  Let $H$ be the $n$ consecutive boundary edges aside from $P_0P_1$ and $P_0P_{n+1}$.  If there is a matching, then some two couples of vertices in $D$ are matched. One of them must be $P_kP_l$ with $k\neq{0,1,n+1}$.
If $l\neq{0}$ and $\wlogg$ $k<l$, then the set $C=P_{k}P_{k+1}...P_{l}$ is convex and has no interior points, thus some boundary edge $P_{i}P_{i+1}$ of $C$ must be matched (which is impossible because that edge belongs to $H$) or we have a single unmatched vertex (it cannot be matched without crossing $P_kP_l$).
If $l=0$, then we must have another couple $P_mP_r$ with $m,r\neq{0,1,n+1}$ and we can apply the same argument to this pair of vertices. \\
Thus we have eliminated all matchings.

\section{Upper Bound}
\label{sec:upper bound}
For which configuration of the vertices of $\graphA$ is $h(\graphA)$ maximal?
In any complete geometric graph $\graphA$, the removal of all the neighbors of some point leaves no matching in the remaining graph, so for any such $\graphA$ on $2n$ points we have $h(\graphA)\leq{2n-2}$. It turns out that this bound is attained.

\begin{theorem}
There exists a complete geometric graph $\graphD$ on $2n$ points in general position in the plane such that for any subgraph $H$ of $\graphD$ with $2n-2$ edges, there is always a perfect matching in $\graphD-H$.
\end{theorem}

\begin{proof}
Define $\graphD$ to be the complete graph on the following $2n$ points: The set $C=P_{1}P_{2}...P_{2n-1}$ of $2n-1$ points in convex position, plus one more point $P$ such that there is no segment $PP_i$ that crosses the interior of $C$ [see figure 6].
For each $i=1,2..,2n-1$ define a perfect matching $M_i$ as follows:\\
$$P_{i-1}P_{i+1},P_{i-2}P_{i+2}..., PP_i$$ It is easy to check that we obtain non-crossing perfect matchings in this manner.  All such matchings are edge-disjoint, because from any segment we can recover the matching to which it belongs. Thus we have $2n-1$ disjoint matchings. Therefore, we need to remove at least $2n-1$ edges to eliminate all possible matchings.
\end{proof}
We thus conclude that $h(\graphD)=2n-2$.

\section{The Random Case}
\label{sec:random case}
In the previous sections we have seen that $n-1\leq{h(\graphA)\leq{2n-2}}$. The examples of $\graphB$ (the convex case) and $\graphD$ (the configuration of the last section) proved that these bounds are tight. Yet, one may still wonder what "usually" happens - in the random case - for example when the points are picked randomly in a uniform distribution (with respect to area) inside a bounded convex region.
The following theorem, our second main theorem in this paper, gives a lower bound which is much higher than the general lower bound of $n-1$.
From now on, $\Gamma$ will stand for a bounded, closed and convex set in the plane with piecewise smooth boundary.

\begin{theorem}\label{thm:lowerbound}
Let $\Gamma$ be as above. Let $X$ be a set of $2n$ points which were picked randomly and independently in a uniform distribution (with respect to area) inside $\Gamma$. Then as $n$ tends to infinity, almost surely we need to remove more than $n+\frac{n}{3\log {(2n)}}$ edges to eliminate all non-crossing matchings tends to 1.  That is, almost always and in average the value of $h(\graphA)$ is more than ${n+\frac{n}{3\log {(2n)}}}$
\end{theorem}

In order to prove this theorem we shall state the following theorem, whose proof will be given in the next section.
\begin{theorem} \label{thm:convexunit}
Let $\Gamma$, $X$ be as above. Suppose $k\geq{\log (n)}$. Then with probability that tends to 1 as $n$ tends to infinity, for every subset of $X$ consisting of $k$ points in convex position there are at least $k$ other points from $X$ which lay inside the convex polygon that the former points describe.
\end{theorem}

We now deduce Theorem \ref{thm:lowerbound} from Theorem \ref{thm:convexunit}:
\begin{proof}
Suppose we could have eliminated all matchings by removing $n+k$ edges where $k<\frac{n}{2+2\log {(2n)}}$.
Color the removed edges red and assume that all matchings in the remaining graph have been eliminated. We replace each red edge by a red midpoint (the middle of the corresponding edge).  We have $n+k$ red points.  The original set of $2n$ points we color blue. By the Ham Sandwich Theorem we can divide the plane with a straight line $l$, which is not parallel to any segment which connects two colored points, such that there are the same number of blue points on each side of $l$ and the same number of red points on each side of $l$ . Note that the points may not be in general position, yet we can solve this problem by a slight perturbation of the locations of the red points. It is also possible, yet has probability $0$ that the blue points themselves are not in general position.

We need to consider two cases:\\
$(i).$ If $n$ is even we then have $n$ blue points and at most $\lfloor{\frac{n+k}{2}}\rfloor$ red points on each side of $l$.\\
$(ii).$ If $n$ is odd we continuously move $l$ to the right in a direction perpendicular to it, until we meet a blue point. Denote this translation of $l$ by $l_1$. We do the same to the left direction, and we denote by $l_2$ the corresponding translation of $l$. On $l_1$ (and similarly on $l_2$) there is exactly one blue point.
Without loss of generality we assume that $dist(l_1,l)\leq{dist(l_2,l)}$.\\
We now divide the plane into two halfplanes whose common boundary is the line $l_1$ and consider each halfplane separately. We erase all the red points between $l_1$ and $l$. Note that they connect blue points left of $l_1$ with points right of $l_2$. Thus in $RP(l_1)$ we have $n+1$ points and at most $\lceil{\frac{n+k}{2}}\rceil$ red points (it is possible that we have one red point on $l$) and in $LP(l_1)$ we have $n-1$ points and at most $\lfloor{\frac{n+k}{2}}\rfloor$ red points.

We continue acting in this manner:\\
After the $m^{th}$ step we are left with at most $2^m$ convex regions, we divide each using the Ham Sandwich theorem and then if we have an even number of blue points in every side, we do nothing. Otherwise we repeat the steps taken in case $(ii)$ above.
We stop dividing a region when it contains less than $4\log 2n$ blue points (and thus not less than $2\log 2n$). \\
We are now going to show that in each such region, we can apply Theorem \ref{thm:hGc1isn} to find a perfect matching.\\
Denote by $B(m),R(m)$ the number of blue and red points, respectively, in an arbitrary region which is created after exactly $m$ steps. In fact we consider an arbitrary sequence of nested regions, each is obtained from the previous one as one of the two halves constructed in the procedure described above. Denote by $H(m)=R(m)-\frac{B(m)}{2} .$ Then
following the process in each step we have:

$B(0)=2n, B(m+1)\geq{\frac{B(m)}{2}-1}$\\
The recursion yields:
$${\frac{2n}{2^m}+2}\geq B(m)\geq{\frac{2n}{2^m}-2}$$

Similarly $\frac{R(m)+1}{2}\geq R(m+1) \geq \frac{R(m)-1}{2}$, and $H(m+1)\leq \frac{H(m)+1}{2}$

As $H(0)=n+k-n=k$, the recursion yields $H(m)\leq \frac{k}{2^m}+1$.

In order to use Theorems \ref{thm:convexunit} and \ref{thm:hGc1isn} we would like to find an integer $m$ such that:\\
$B(m)\geq{2\log {2n}}$ and $H(m)\leq{1}$.\\

For the first condition we can should choose $m$ such that $\frac{2n}{2^m}-2\geq{2 \log {2n}}$,
and for the second condition we want $m$ to satisfy that $\frac{k}{2^m}+1 < 2$.\\
Combining the two constraints gives: $$k < {2^m}\leq{\frac{2n}{2+2\log {(2n)}}}$$ Such an integer $m$ exists if
$${\frac{2n}{2+2\log {(2n)}}}>2k\Leftrightarrow {\frac{n}{2+2\log {(2n)}}}>k$$
But this is exactly what we assumed in the theorem.\\
Denote the integer $m$ that we have chosen by $m^{*}$.  Then after $m^{*}$ steps we have decomposed the plane into convex regions, each containing at least $2 \log {2n}$ points, and the difference between the number of red points and half of the number of blue points is at most 1. By Theorem \ref{thm:convexunit}, with probability which tends to 1 as $n$ tends to infinity, the number of blue points on the boundary of the convex hull of the blue points in each region is at most half of the total number of the blue points in that region.
We can now apply Theorem \ref{thm:hGc1isn} to each final region 
in order to find a perfect matching in each region (note that we always have an even number of points in each region, at each step). Because all regions are convex and disjoint, we have found a non-crossing matching without using the red edges. This completes the proof.
\end{proof}

\subsection{Random convex sets}\label{sec:randomconvex}
This subsection is devoted to the proof of Theorem \ref{thm:convexunit}. We first describe the idea behind the proof, and then the strategy of the proof
. Let $n$ be a large integer, $k$ be another integer, not less than $\log n$ and $\varepsilon$ be a small number that we shall specify later . Let $X$ be a set of $n$ points, chosen i.i.d from $\Gamma$ with respect to uniform distribution. We show that with a probability which tends to $0$ as $n$ tends to infinity there exists a subset of $X$, containing $k$ points in convex position s.t. their convex hull is of area smaller then $\varepsilon$. In addition we show that the probability that there exists a subset of $X$, containing $k$ points in convex position s.t. their convex hull is of area larger then $\varepsilon$ and that there are less than $k$ other points of $X$ in the interior of the convex hull also tends to $0$ as $n$ tends to infinity. Thus, with high probability, any subset of $k$ points in convex position contains inside its convex hull at least $k$ other points from $X$.

The strategy of the proof will be to introduce a square lattice with small enough scale such that every convex subset of $\constK$ of size larger than $\varepsilon$ lays in the interior of a lattice triangle whose area is no more than $A\varepsilon$, with $A$ some universal positive constant. We shall use a well known theorem from the theory of random convex sets to show that the probability that in one of the relevant lattice triangles there are $k$ points in convex position is very small.
We then exclude the second option. We show that every convex subset of $\constK$ of size larger than $\varepsilon$ contains a lattice triangle of size at least $a\varepsilon$, for another positive universal constant $a$. We show that with very high probability inside every such triangle there are at least $k$ points from $X$.

Before stating the theorem we need from the theory of random convex sets, we start with some definitions and notations.
Denote by \\ $\la = \{(\frac{a}{M},\frac{b}{M})|a,b\in\mathbb{Z}\}$ the lattice with basic square of size $\frac{1}{M}\times{\frac{1}{M}}$.\\
Let $\tm$ be the set of all triangles with vertices in $\la$. \\
Let $P(n,\constK)$ denote the probability that $n$ points in a bounded, convex, closed planar set with piecewise smooth boundary $\constK$, chosen independently with uniform distribution w.r.t. area, lie in convex position.

We now state the following key fact that we shall need for proving Theorem \ref{thm:convexunit}. For details see for example \cite{BA2}, \cite{VALTR} or \cite{BLAS}.
\begin{theorem}\label{thm:barany}
There exist two universal constants $B_1 , B_2 $ such that for any closed bounded convex set $\Gamma$ and for all $n$ the following inequality is satisfied:\\
$B_1 < (n^2)\sqrt[n]{P(n,\constK)}<B_2$.
\end{theorem}

We also need some geometric lemmas that we shall now state and prove in Appendix \ref{app:cls}.
\begin{lemma}\label{lem:bounding cls}
Let $\Gamma$ be a closed bounded convex set in the plane, with diameter $1$
. Let $M>0$ be an integer. Then for any closed convex subset $K$ of $\Gamma$, with area $S$ we have that if $S > 100/M$, then there exists a triangle $T\in\tm$ which contains $K$ and whose area is no more than $100S$. If $S \leq 100/M$, then there exists a triangle $T\in\tm$ which contains $K$ and whose area is no more than $10000/M$
\end{lemma}
\begin{lemma}\label{lem:bounded cls}
Let $\Gamma, K, M, S$ be as before. If $S > 100/M$, then there exists a triangle $T\in\tm$ which is contained in $K$ and whose area is not less than $S/100$.
\end{lemma}
\begin{lemma}\label{lem:not to big}
Let $M, \Gamma$ be as before. Then there exists a square $\mathcal{D}$ of side length $100$ whose vertices are in $\la$  such that for any closed convex subset $K\subseteq \Gamma$,  the corresponding triangle $T\in\tm$ that we construct in Lemma \ref{lem:bounding cls} lays in the interior $\mathcal{D}$.
\end{lemma}
We denote by $\tmgamma$ the set of triangles from $\tm$ which are contained in the interior of the triangle $\mathcal{D}$ constructed in the above lemma.
Note that there are $O(M^6)$ triangles in $\tmgamma$.

We are now ready to prove Theorem \ref{thm:convexunit}.
\begin{proof}
(Proof of theorem \ref{thm:convexunit}).
The strategy will be to find an $\varepsilon$ such that with very high probability there are no $k$ points in convex position that bound an area of at most $\varepsilon$, and on the other hand, with very high probability there are no $k$ points such that the interior of their convex hull contains a small number of points (less than $k$) and bounds an area of more than $\varepsilon$. As a conclusion we have that the probability there is a subset of $k$ points in convex position with less than $k$ other points inside its convex hull is very low.

Let $X=\{x_1,...,x_{n}\}$ be a set of $n$ points in $\Gamma$, chosen uniformly and indepndently with respect to area. Denote $[n]={\{1,2,..,n\}}$. let $k$ be $\lceil\log (n)\rceil$. Choose $M$ to be $n$, $\varepsilon = \frac{10 k}{n}$.
For any $J\subseteq{[n]}$, let $X_J= \{x_i|i\in{J}\}$. Denote by $C_J$ the event that $X_J$ is a convex set of points, and denote by $A_{J,\varepsilon}$ the event that the area of the convex hull of $X_J$ is at most $\varepsilon$.
For any  triangle $T$, denote by $B_{J,T}$ the event that all points of $X_J$ lie in $T$. Finally, denote by $c_J$, $a_{J,\varepsilon}$, and $b_{J,T}$ respectively the probabilities of these events. \\
Let $P(k,\varepsilon)$ be the probability that there exists a set of $k$ points in $X$ which are in convex position and whose convex hull has area at most $\varepsilon$.

By the union bound,

$$P(k,\varepsilon)\leq{\sum_{J\subseteq{[n]},|J|=k}{P(C_J\cap{A_{J,\varepsilon}})}}.$$
\\
$\log {n}\rightarrow \infty$ as $n\rightarrow\infty$, therefore we see that for large enough $n$, $\varepsilon > 100/M$.
Hence, by Lemma \ref{lem:bounding cls}, for any $J\subseteq{[n]}$ such that the area bounded by the convex hull of $X_J$ is smaller than $\varepsilon$ there is some $T$ in $\tm$ with $\area(T)\leq{100\varepsilon}$ and $X_J\subseteq{T}$. Moreover, Lemma \ref{lem:not to big} tells us that $T\in\tmgamma$.
Thus we get for any $|J|=k$,
$${P(C_J\cap{A_{J,\varepsilon}})}\leq\sum_{T\in{\tmgamma},\area(T)\leq 100\varepsilon} P(C_J\cap{B_{J,T}}).$$

But by the Theorem \ref{thm:barany}, $P(C_J\cap B_{J,T})\leq {{{B_2}/{k^2}}^k}{area(T)^k}$. Thus, \\
$${P(C_J\cap{A_{J,\varepsilon}})}\leq{\sum_{T\in \tm,\area(T)\leq 100\varepsilon} (\frac{B_2}{k^2})^k(\area(T))^k\leq LM^6(\varepsilon)^k}(\frac{B_2}{k^2})^k.$$

Where $L$ is some positive constant. The last inequality follows from the fact that there are $O(M^6)$ triangles in $\tmgamma$. Note that we have used the fact conditioning over some fixed domain inside $\Gamma$ gives rise to uniform samples from this domain.\\
Hence, $P(k,\varepsilon)\leq \sum_{J\subseteq{[n]},|J|=k} L M^6 {100 \varepsilon}^k (\frac{B_2}{k^2})^k = \binom{n}{k} L M^6{100 \varepsilon}^k (\frac{B_2}{k^2})^k .$ \\
By using standard binomial bounds, we see that the latter expression is bounded by

$\leq L(\frac{ne}{k})^k n^6 (\frac{10000 k^2}{n})^k(\frac{B_2}{k^2})^k =
L n^6 (\frac{10000 B_2 e}{k})^k \stackrel{n\longrightarrow \infty}{\longrightarrow}{0} .$
Where we have substituted our choices of $M,k,\varepsilon$ as functions of $n$.

For a triangle $T$ and a positive integer $k$, denote by $D_{T,k}$ the event that at most $k$ points of $X$ lay in the interior of the $T$. Denote by $d_{T,k}$ the probability of this event. In the case that event $D_{T,k}$ occurs, there are at least $n-k$ points out of $T$. Thus,
$$d_{T,k} \leq \binom{n}{k}(1-\\area(T))^{n-k}$$
Let $Q(k,\varepsilon)$ be the probability that there exist $k$ points in $X$ in convex position which bound an area of at least  $\varepsilon$ and whose interior contains at most $k$ other points of $X$.\\
According to Lemma \ref{lem:bounded cls} for any $J\subseteq[n],|J|=k$ such that $X_J$ is in convex position and bounds an area of at least $\varepsilon$ there exists a triangle $T\in \tm$ in the interior of the convex hull of $X_J$ with $\area(T)\geq{\frac{\varepsilon}{100}}$. As $X_J$ has at most $k$ points in its interior, $T$ also has at most $k$ points in its interior as well. We conclude:
\bae
Q(k,\varepsilon) &\leq \sum_{T \in \tmgamma,\area(T)
\geq \frac{\varepsilon}{100}} d_{T,k} \\
&\leq \sum_{T \in \tmgamma, \area(T)
\geq \frac{\varepsilon}{100}} \binom{n}{k} (1-\area(T))^{n-k}\\
&\leq{L' M^6{\binom{n}{k}(1-\frac{\varepsilon}{100})^{n-k}}}.\\
\eae
Again $L'$ is some constant.
With the same choices of $M,k,\varepsilon$ as before and using standard estimates that the last expression we have obtained is approximately
\bae \nonumber
(\frac{ne}{k})^{k} L' {n^6}{(1-\frac{k^2}{n})}^{n-k} &\leq{L' (\frac{ne}{k})^{k}{n^6}{(1-\frac{k^2}{n})^{n}{2^k}}}  \\
& \leq{L' (\frac{2ne}{k})^{k}{n^6}{\frac{1}{e^{k^2}}}} \\
& = L' \frac{n^k(2e)^kn^6}{e^{k^2}k^k}.\\
\eae
As $k = k(n)\geq{\log n}$, $\frac{n^k}{e^{k^2}}\leq{1}$ and thus the last expression is bounded by:

$$\leq{\constD \frac{(2e)^kn^6}{k^k}}\stackrel{n\longrightarrow \infty}{\longrightarrow}{0}.$$

Now let $C(k)$ be the probability that there exists a set $J\subseteq[n],|J|=k$ such that the points of $X_J$ are in convex position, and the interior of their convex hull contains $k$ other points in its interior. As before we assume that $k = \lceil{\log n}\rceil$. Also we take $\varepsilon$ as before. By definition it is clear that $C(k)\leq{P(k,\varepsilon)+Q(k,\varepsilon)}$. But we have showed that
$P(k,\varepsilon),Q(k,\varepsilon)\stackrel{n\longrightarrow \infty}{\longrightarrow}0$, and thus our proof is complete.
\end{proof}

We can now use the method of coverings by lattice triangles which we have developed to bound the expected number of empty $k$-gons ($k$-gons without any other points from the set in the interior of their convex hull).
Indeed, as a consequence of Theorem \ref{thm:convexunit} we have:
\begin{corollary}
Let $\constK$ as above. Let $X$ be a set of $n$ points chosen with uniform distribution in $\Gamma$. With probability that tends to 1 as $n$ tends to infinity, if $k\geq{\log n}$ there are no empty convex k-gons whose vertices are in $X$.
\end{corollary}

\section{random upper bound}
\label{sec:random upper}

Next we state an upper bound for the random case and sketch the proof:\\
\begin{definition}\label{def:exp_conv}
For any bounded, closed, convex set in the plane, with piecewise smooth boundary $\Gamma$, denote by \emph{$C(\constK,n)$} the expected number of points of which are the vertices of the boundary of the convex hull of a random set of $n$ points in $\Gamma$.
\end{definition}

\begin{claim}
Let $\constK$ be as in Definition \ref{def:exp_conv}. Let $X$ be a set of $2n$ points chosen with uniform distribution in $\constK$. Then there exists some $s>0$ such that, as $n$ tends to infinity, the probability tends to 1 that the removal of some $2n-sC(\constK,2n)$ edges will eliminate all non-crossing matchings.
\end{claim}
This bound depends on $C(\constK,n)$. It is known that $C(P_m,n)=\Theta(m \log n)$ where $P_m$ is an $m$-gon, and $C(Cr,n)=\Theta(\sqrt[3]{n})$ where $Cr$ is a circle. [see \cite{HAR}]\\
\begin{proof}
\emph{[Proof Sketch]} Let $P_0$ be a point on the boundary of the convex hull of $\constK$. For any other point $P$ on that boundary, the segment $PP_0$ splits $X$ into two parts such that any segment between two points from different parts crosses $PP_0$. If the number of the points on each side of $PP_0$ is odd then $PP_0$ is not a part of any perfect non-crossing matching. \\
Let $O(X)=\{P\in X|  P$ is on the boundary of the convex hull and $PP_0$ splits $X$ into odd parts$\}$. If we remove all edges connecting $P$ to any other point in $X-O(X)$, then clearly we eliminate all matchings.  For large enough $n$ ,  $E(|O(X)|)=\frac{C(K,n)}{2}$.  With probability that tends to 1 for some $s>0$ , $|O(X)|\geq{s{C(K,n)}}$ and we can remove $|X-O(X)|\leq{2n-s{C(K,n)}}$ edges and eliminate all matchings.
\end{proof}

\newpage

\newpage
\appendix
\section{Convex Lattice sets}\label{app:cls}
The proofs in this appendix will be somehow sketchy, as these lemmas are quite standard. In addition, the constants we use in these lemmas are far from being optimal (yet suffice for our goals). We begin by
proving Lemma 3.4. The proof of Lemma 3.5 uses a similar technique (though some technical details
are different) and hence will be omitted. We prove a slightly stronger version.
\begin{lemma}
Let $\Gamma$ be a closed bounded convex set in the plane, with diameter $1$. Let $M>0$ be an integer. Then for any closed convex subset $K$ of $\Gamma$, with area $S$ we have that if $S > 100/M$, then there exists a triangle $T\in\tm$ which contains $K$ and whose area is no more than $64S$. If $S \leq 100/M$, then there exists a triangle $T\in\tm$ which contains $K$, whose area is no more than $6400/M$ and whose diameter is no more than $12$.
\end{lemma}
\begin{proof}
As we shall be interested in large $M$, we shall assume for convenience that $M$ is larger than $100/area(\Gamma)$. 

Without loss of generality we assume that $S \geq 100/M$, otherwise we bound $K$ in a closed convex subset of $\Gamma$ with that size.
Let $x,y$ be two points of $K$ with distance $d(x,y) = diam(K)$. Such two points exist, as $K$ is compact. Denote by $l$ the segment between $x$ and $y$, denote by $n_1 ,n_2 $ the two orthogonal line to $l$ at the points $x,y$ respectively. Let $A,B' \in n_1 , D', C' \in n_2$ be the four vertices of the smallest rectangle whose edges are parallel to $l, n_1$ which contains $K$ (in cyclic order). Due to convexity, it is easy to see that its area is no more than twice the area of $K$, and that its diameter is no more than $\sqrt{2}$. Note that convexity reasoning also shows that its shortest edge is at least of length $100/M$. Indeed, if we consider points $z,w$ on either sides of $l$ where $B'D', C'A$ touch $K$. Due to convexity the area of $xzyw$ is at least half of that of $K$. Thus, the sum of altitudes from $z,w$ to $l$ is at least $100/M$.

Now we may bound $AB'C'D'$ in a triangle $ABC$ whose area is twice the area of the rectangle, such that $B'$ is the middle of $AB$, and $C'$ is the middle of $AC$.
Again, the diameter of $ABC$ is no more than $2 \sqrt{2}$, and its shortest edge is of length at least $200/M$. Its area is between twice the area of $K$ and $4$ times this area.
Thus, it is sufficient to prove the theorem for the reduced situation.

Area considerations show that the radius of the inscribed circle of $ABC$ is twice the ratio of its area and its perimeter, and therefore it equals at least $\frac{400/M}{6 \sqrt{2}}$. It is easy to verify that any such circle contains a lattice point in its interior.

Now, reflect the triangle $ABC$ through $A$, and obtain a triangle $A B_A C_A$. Denote by $A_2$ the reflection of $A$ in the line $B_A C_A$. Inside the inscribed circle of $A B_A C_A$ there is a lattice point $A_1$. Similarly one can define the points $B_1 , B_2 , C_1 , C_2$. Clearly $ABC$ is contained in $A_1 B_1 C_1$, which, in its turn is contained in $A_2 B_2 C_2$. The last triangle is similar to $ABC$, with ratio 4:1, thus, its diameter, which is larger than that of $A_1 B_1 C_1$, is no more than $8 \sqrt{2} < 12$, and its area, which is again larger than the area of $A_1 B_1 C_1$, is no more than $64 area(K)$. Thus, $A_1 B_1 C_1$ is the required triangle.

\end{proof}

We now prove Lemma \ref{lem:not to big}:
Let $M, \Gamma$ be as in Subsection \ref{sec:randomconvex}. Then there exists a square $\mathcal{D}$ of side length $100$ whose vertices of in $\la$  such that for any closed convex subset $K\subseteq \Gamma$,  the corresponding triangle $T\in\tm$ that we construct in Lemma \ref{lem:bounding cls} lays in the interior $\mathcal{D}$.
\begin{proof}
First, the diameter of any triangle constructed in Lemma \ref{lem:bounding cls} is no more than $12$.
Thus, Choose any point $O$ of $\Gamma$, then any point of the last triangle is within a distance of no more than $13$ from $O$. Hence lays in a disk of radius $11$ around $O$. We can easily find a lattice square of side $100$ containing this disk. As this disk did not depend on $K$, we are done.
\end{proof}

\newpage

\begin{figure}[]
\begin{center}
\section{Figures}\label{sec:fig}
$$\scalebox{0.7}{\includegraphics{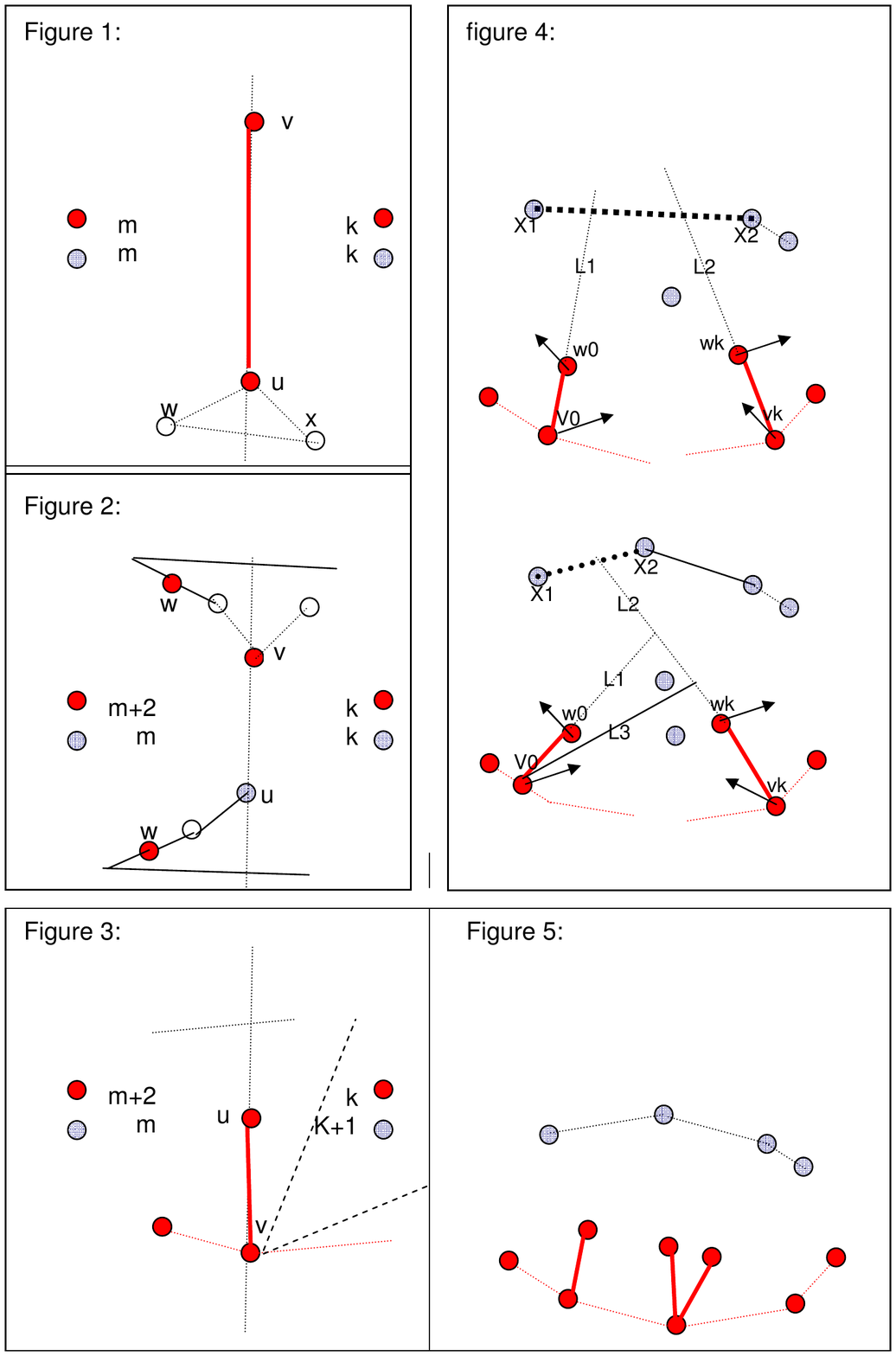}}$$
\end{center}
\end{figure}

\begin{figure}[]
\begin{center}

$$\scalebox{0.7}{\includegraphics{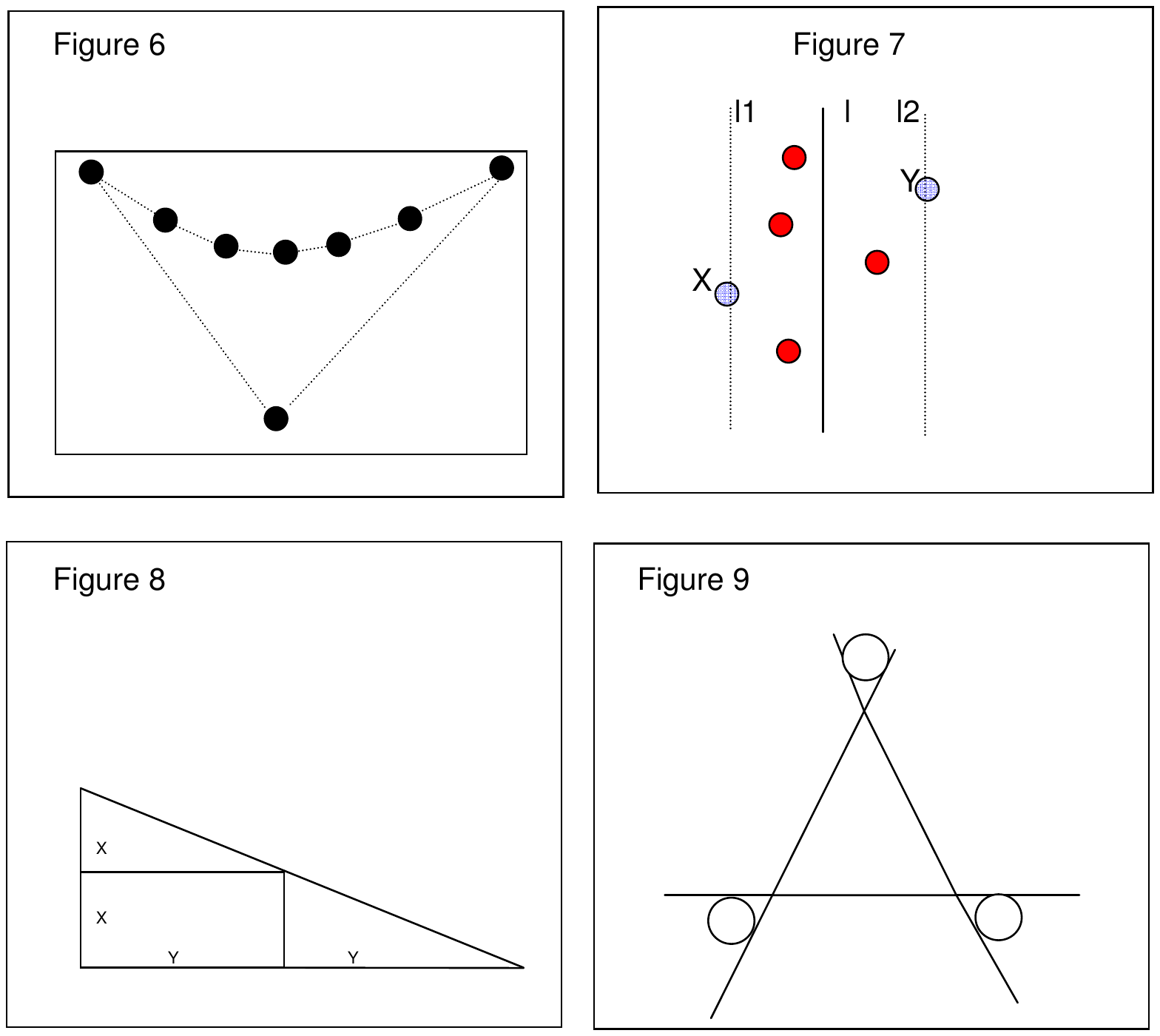}}$$

\end{center}
\end{figure}

\end{document}